\documentclass[12pt,a4paper]{article}
\usepackage[top=3cm, bottom=3cm, left=3cm, right=3cm]{geometry}
\usepackage[latin1]{inputenc}
\usepackage{amsmath}
\usepackage{amsthm}
\usepackage{amsfonts}
\usepackage{amssymb}
\usepackage{graphics}
\usepackage[hang,flushmargin]{footmisc} 

\usepackage{amsmath,amsthm,amssymb,graphicx, multicol, array}
\usepackage{enumerate}
\usepackage{enumitem}
\usepackage{xcolor}
\usepackage{currfile}

\usepackage[bookmarks,colorlinks,breaklinks]{hyperref}

\hypersetup{linkcolor=blue,citecolor=red,filecolor=blue,urlcolor=blue} 

\usepackage{tabto}

\usepackage{tikz}
\usepackage{pgfplots}

\newtheorem{thm}{Theorem}[section]
\newtheorem{lem}{Lemma}[section]

\newtheorem{exa}{Example}[section]

\newtheorem{dfn}{Definition}[section]
\newtheorem{exe}{Exercise}[section]

\newcommand{\N}{\mathbb{N}}
\newcommand{\Z}{\mathbb{Z}}
\newcommand{\Q}{\mathbb{Q}}
\newcommand{\R}{\mathbb{R}}
\newcommand{\C}{\mathbb{C}}

\usepackage{tocloft}

\usepackage{fancyhdr}
\title{Irrationality Measure of Pi}
\date{}
\author{N. A. Carella}

\begin{document}
\thispagestyle{empty}
\date{}

\maketitle
\textbf{\textit{Abstract}:} The first estimate of the upper bound  $\mu(\pi)\leq42$ of the irrationality measure of the number $\pi$ was computed by Mahler in 1953, and more recently it was reduced to $\mu(\pi)\leq7.6063$ by Salikhov in 2008. Here, it is shown that $\pi$ has the same irrationality measure $\mu(\pi)=\mu(\alpha)=2$ as almost every irrational number $\alpha>0$.\let\thefootnote\relax\footnote{ \today \date{} 
}

\tableofcontents
\newpage
\section{Introduction} \label{s597}
Let $\alpha \in \R$ be a real number. The irrationality measure $\mu(\alpha)$ of the real number $\alpha$ is the infimum of the subset of real numbers $\mu(\alpha)\geq1$ for which the Diophantine inequality
\begin{equation} \label{eq597.036}
  \left | \alpha-\frac{p}{q} \right | \ll\frac{1}{q^{\mu(\alpha)} }
\end{equation}
has finitely many rational solutions $p$ and $q$. The analysis of the irrationality measure $\mu(\pi)\geq 2$ was initiated by Mahler in 1953, who proved that 
\begin{equation} \label{eq597.38}
  \left | \alpha-\frac{p}{q} \right | >\frac{1}{q^{42} }
\end{equation}
for all rational solutions $p$ and $q$, see \cite{MK53}, et alii. This inequality has an effective constant for all rational approximations $p/q$. Over the last seven decades, the efforts of several authors have improved this estimate significantly, see Table \ref{t9001}. More recently, it was reduced to $\mu(\pi)\leq7.6063$, see \cite{SV08}. This note has the followings result.
\begin{thm} \label{thm597.21} For any number $\varepsilon >0$, the Diophantine inequality
\begin{equation} \label{eq597.46}
  \left | \pi-\frac{p}{q} \right | \ll\frac{1}{q^{2+\varepsilon }}
\end{equation}	
has finitely many rational solutions $p$ and $q$. In particular, the irrationality measure $\mu(\pi)=2.$
\end{thm}

After some preliminary preparations, the proof of Theorem \ref{thm597.21} is assembled in Section \ref{s944}. Three distinct proofs from three different perspective are presented. Section \ref{s590} explores the relationship between the irrationality measure of $\pi=[a_0;a_1,a_2, \ldots ]$ and the magnitude of the partial quotients $a_n$. In Section \ref{s544} there is an application to the convergence of the Flint Hills series. A second independent proof of the convergence of this series also appears in Section \ref{s5444}. Some numerical data for the ratio $\sin p_n/ \sin 1/p_n$ is included in the last Section. These data match the theoretical result.

\begin{table}[h!]
\centering
\caption{Historical Data For $\mu(\pi)$} \label{t9001}
\begin{tabular}{l|l|l}
Irrationality Measure Upper Bound&Reference&Year\\
\hline
 $\mu(\pi) \leq 42$ & Mahler, \cite{MK53}&1953\\
 $\mu(\pi) \leq 20.6$ & Mignotte, \cite{MM74}&1974\\
 $\mu(\pi) \leq  14.65$ & Chudnovsky, \cite{CG82}&1982\\
$\mu(\pi) \leq 13.398$ & Hata, \cite{HM93}&1993\\
$\mu(\pi) \leq 7.6063$ & Salikhov, \cite{SV08}&2008\\
\end{tabular}
\end{table}

\section{Notation} \label{s0000}
The set of natural numbers is denoted by $\N=\{0,1,2,3,\ldots \}$, the set of integers is denoted by $\Z=\{\ldots -3,-2,-1,0,1,2,3,\ldots \}$, the set of rational numbers is denoted by $\Q=\{a/b: a,b \in \Z \}$, the set of real numbers is denoted by $\R=(-\infty,\infty)$, and the set of complex numbers is denoted by $\C=\{x+iy:x,y \in \R\}$.
For a pair of real valued or complex valued functions, $f,g:\C \longrightarrow \C$, the proportional symbol $f \asymp g$ is defined by $ c_0g \leq f \leq c_1 g$, where $ c_0, c_1 \in \R$ are constants. In addition, the symbol $f \ll g$ is defined by $|f|\leq c|g|$ for some constant $c>0$. \\

\section{Main Result} \label{s944}
The analysis of the irrationality measure $\mu(\pi)$ of the number $\pi$ provided here is completely independent of the prime number theorem, and it is not related to the earlier analysis used by many authors, based on the Laplace integral of a factorial like function
\begin{equation} \label{eq944.121}
\frac{1}{i 2\pi} \int_{C}\left (\frac{n!}{z(z-1)(z-2)(z-3)\cdots (z-n)} \right )^{k+1} e^{-tz} dz,
\end{equation}
where $C$ is a curve around the simple poles of the integrand, and $k \geq 0$, see \cite{MK53}, \cite{MM74}, \cite{CG82}, \cite{HM93}, \cite{SV08}, and \cite{BF00} for an introduction to the rational approximations of $\pi$ and the various proofs. The improvements made using rational integrals and the prime number theorem, as the new estimate in \cite{ZZ19}, are limited to small incremental improvements, and no where near the numerical data, see Table \ref{t9007}. Continuing at this pace, it will take another 30 or more years to prove the true irrationality measure of $\pi$, which is $\mu(\pi)=2$ as approximated by the numerical data. These numerical approximations are not examples of the \textit{Strong law of small number}, confer \cite{GR1988} for details. Here, the numerical data consists of very large numbers, including random evaluations.

\subsection{First Proof}
The first proof is based on elementary properties of the sine function and basic Diophantine analysis. 
\begin{proof}[\textbf{Proof}] ({\bfseries Theorem \ref{thm597.21}}) Assume $\mu=\mu(\pi)>2$. Let $\{p_n/q_n : n \geq 1\}$ be the sequence of convergents of the number $\pi$. Setting $z=|\pi q_n-p_n|$ in the inequality 
\begin{equation}\label{eq944.40}
z-\frac{z^3}{6}\leq \sin z\leq z,    
\end{equation}
where
$|z|<1$, yields the inequality
\begin{equation} \label{eq944.50}
\left |\pi q_n-p_n \right |\ll \left | \sin (\pi q_n-p_n )\right | \ll \left |\pi q_n-p_n \right |
\end{equation}
for all large integers $q_n\geq 1$. Observe, that line \eqref{eq944.50} clearly shows that the lower bound of the sine function is independent of the irrationality measure $\mu=\mu(\pi)$ of the irrational number $\pi\ne0$. \\

In addition, the Dirichlet approximation inequality 
\begin{equation} \label{eq944.55}
\frac{1}{q_n^{\mu-1}}\ll \left | \pi q_n-p_n \right | \ll \frac{1}{q_n}
\end{equation}
holds for all large integers $q_n\geq 1$. \\

Combining \eqref{eq944.50} and \eqref{eq944.55} yield
\begin{equation} \label{eqkkkkc}
\left | \pi q_n-p_n \right |\ll\frac{1}{q_n}  \ll \left | \sin (\pi q_n-p_n )\right |    \ll \left | \pi q_n-p_n \right |\ll\frac{1}{q_n}
\end{equation}
for all large integers $n\geq 1$. On the contrary, 
\begin{equation} \label{eq944.60}
\left | \pi q_n-p_n \right |\ll \frac{1}{q_n^{\mu-1}} \ll \left | \sin (\pi q_n-p_n )\right |\ll \left | \pi q_n-p_n \right | \ll\frac{1}{q_n}
\end{equation}
for all large integers $q_n\geq 1$. But, this contradicts the hypothesis $\mu(\pi)>2$. Therefore, $\mu(\pi)=2$
\end{proof}

\subsection{Second Proof}
The second proof is based on a result for the upper bound of the reciprocal sine function over the sequence of $\{p_n: n\geq 1\}$ derived in Section \ref{s5534}. This is equivalent to a result in Section \ref{s5599}.
\begin{proof}[\textbf{Proof}] ({\bfseries Theorem \ref{thm597.21}}) Let $\varepsilon >0$ be an arbitrary small number, and let $\{p_n/q_n: n \geq 1\}$ be the sequence of convergents of the irrational number $\pi$. By Theorem \ref{thm6699.300}, the reciprocal sine function has the upper bound
	\begin{equation} \label{eq944.883}
	\left |\frac{1}{\sin \pi^2q_n}\right |\ll q_n^{1+\varepsilon} .
	\end{equation} 
	Moreover, $\sin (\pi^2q_n)= \sin \left (\pi^2q_n -\alpha p \right )$ if and only if $\alpha p=\pi p_n$, where $p$ and $p_n$ are integers. These information lead to the following relation.  
	\begin{eqnarray} \label{eq944.885}
	\frac{1}{q_n^{1+\varepsilon}}&\ll&  \left |\sin \left (\pi^2q_n \right )\right |\\
	&\ll& \left |\sin \left (\pi^2q_n -\pi p_n \right )\right |\nonumber\\
	&\ll& \left |\pi q_n-p_n \right |\nonumber
	\end{eqnarray}
	for all sufficiently large $n$. Therefore, 
	\begin{eqnarray} \label{eq597.36}
	\left | \pi-\frac{p_n}{q_n} \right | 
	&\gg&\frac{1}{q^{2+\varepsilon} }\\
	&=&\frac{1}{q^{\mu(\pi)+\varepsilon} }\nonumber.
	\end{eqnarray}
	Clearly, this implies that the irrationality measure of the real number $\pi$ is $\mu(\pi)=2$, see Definition \ref{dfn2000.01}. 
	Quod erat demontrandum.
\end{proof}

This theory is consistent with the numerical data in Table \ref{t9007}, which shows the measure approaching 2 as the rational approximation $p_n/q_n \to \pi$.
\subsection{Third Proof}
The third proof is based on the asymptotic expansion of the sine function in Section \ref{s938}. Related results based on the cosine and sine functions appears in \cite{AM11} and \cite{CP19}.
\begin{proof}[\textbf{Proof}] ({\bfseries Theorem \ref{thm597.21}}) Given a small number $\varepsilon >0$, it will be shown that $\psi(q) \leq q^{\mu(\pi)+\varepsilon}\leq q^{2+\varepsilon}$ is the irrationality measure of the real number $\pi$. \\
	
Observe that in Theorem \ref{thm938.40}, the asymptotic expansion of the sine function satisfies
	\begin{equation} \label{eq944.77}
	\left |\sin \left (p_n \right ) \right | \gg \left |\sin \left (\frac{1}{p_n} \right ) \right |.
	\end{equation}
	This inequality leads to the irrationality measures $\psi(q) =q^{2+\varepsilon}$ for real number $\pi $. Specifically, 
	\begin{eqnarray} \label{eq944.85}
	\left |p_n -\pi q_n \right |&\gg&  \left |\sin \left (p_n -\pi q_n \right )\right |\\
	&= &\left |\sin \left (p_n \right )\right |\nonumber \\
	&\gg&  \left |\sin \left (\frac{1}{p_n} \right ) \right | \nonumber\\
	&\gg&\frac{1}{p_n} \nonumber\\
	&\gg&\frac{1}{q_n}\nonumber\\
	&\gg&\frac{1}{q^{1+\varepsilon}}\nonumber
	\end{eqnarray}
	since $p_n=\pi q_n +O\left (q_n^{-1}\right )$. Therefore,
	\begin{equation} \label{eq944.97}
	\psi(q) \leq q^{\mu(\pi)+\varepsilon}\leq q^{2+\varepsilon}.
	\end{equation}
	Quod erat demontrandum.
\end{proof}

\subsection{Four Proof}
The fourth proof is based on the Diophantine inequality
\begin{equation} \label{eq938.101}
\frac{1}{2q_{n+1}q_n} \leq \left | \alpha - \frac{p_n}{q_n}  \right | \leq \frac{1}{q_n^{\mu(\alpha)}} 
\end{equation}
for irrational numbers $\alpha \in \R$ of irrationality measure $\mu(\alpha)\geq 2$, confer Lemma \ref{lem2000.05}, Definition \ref{dfn2000.01}, and a result for the partial quotients in Theorem \ref{thm590.45}.\\

\begin{proof}[\textbf{Proof}] ({\bfseries Theorem \ref{thm597.21}}) Take the logarithm of the Diophantine inequality associated to the real number $\pi$. Specifically,
	\begin{equation} \label{eq938.103}
	\frac{1}{2q_{n+1}q_n} \leq \left | \pi- \frac{p_n}{q_n}  \right | \leq \frac{1}{q_n^{\mu(\pi)}} 
	\end{equation}
	to reach
	\begin{equation} \label{eq944.108}
	\mu(\pi)\leq \frac{\log 2q_{n+1}q_n}{\log q_n}=1+\frac{\log 2q_{n+1}}{\log q_n}.
	\end{equation}
	By Theorem \ref{thm590.45}, there is a constant $c>0$ for which $a_n\leq c$. Hence, 
	\begin{eqnarray} \label{eq944.105}
	q_{n+1} 
	&=&a_{n+1}q_n+q_{n+1} \\
	&\leq&2a_{n+1}q_n\nonumber\\ 
	&\leq&2cq_n\nonumber.
	\end{eqnarray}
	Substitute the last estimate \eqref{eq944.105} into \eqref{eq944.108} obtain
	\begin{eqnarray} \label{eq944.110}
	\mu(\pi) 
	&\leq&1+\frac{\log 4cq_{n}}{\log q_n} \\
	&\leq&2+\frac{\log 4c}{\log q_n} \nonumber
	\end{eqnarray}
	for all sufficiently large $n\geq 1$. Taking the limit yields
	\begin{equation} \label{eq944.128}
	\mu(\pi)=\lim_{n \to \infty} \left( 2+\frac{\log 4c}{\log q_n}\right )= 2.
	\end{equation}
	Quod erat demontrandum.
\end{proof}

\section{Harmonic Summation Kernels}\label{s5534}
The harmonic summation kernels naturally arise in the partial sums of Fourier series and in the studies of convergences of continuous functions.

\begin{dfn} \label{dfn5534.100} The Dirichlet kernel is defined by
\begin{equation} \label{eq5534.200}
\mathcal{D}_x(z)=\sum_{-x\leq n \leq x} e^{i 2nz}=\frac{\sin((2x+1)z)}{ \sin \left ( z \right )},
\end{equation} 
where $x\in \N$ is an integer and $ z \in \R-\pi\Z$ is a real number. 
\end{dfn}

\begin{dfn} \label{dfn5534.102} The Fejer kernel is defined by
\begin{equation} \label{eq5534.204}
\mathcal{F}_x(z)=\sum_{0\leq n \leq x,} \sum_{-n\leq k \leq n} e^{i 2kz}=\frac{1}{2}\frac{\sin((x+1)z)^2}{ \sin \left ( z \right )^2},
\end{equation} 
where $x\in \N$ is an integer and $ z \in \R-\pi\Z$ is a real number. 
\end{dfn}

These formulas are well known, see \cite{KT89} and similar references. For $z \ne k \pi$, the harmonic summation kernels have the upper bounds $\left |\mathcal{K}_x(z) \right |=\left |\mathcal{D}_x(z) \right | \ll |x|$, and $\left |\mathcal{K}_x(z) \right |=\left |\mathcal{F}_x(z) \right | \ll |x^2|$. \\

An important property is the that a proper choice of the parameter $x\geq1$ can shifts the sporadic large value of the reciprocal sine function $1/\sin z$ to $\mathcal{K}_x(z)$, and the term $1/\sin(2x+1)z$ remains bounded. This principle will be applied to the lacunary sequence $\{p_n : n \geq 1\}$, which maximize the reciprocal sine function $1/\sin z$, to obtain an effective upper bound of the function $1/\sin z$.\\

There are many different ways to prove an upper bound based on the harmonic summation kernels $\mathcal{D}_x(z)$ and or $\mathcal{F}_x(z)$. An elementary approach is provided below. \\

The Dirichlet kernel in Definition \ref{dfn5534.100} is a well defined continued function of two variables $x,z \in \R$. Hence, for fixed $z$, it has an analytic continuation to all the real numbers $x \in \R$. \\

\begin{lem}\label{lem5534.505} Let $k\geq 1$ be a small fixed integer, and let $\{p_n/q_n : n \geq 1\}$ be the sequence of convergents of the real number $\pi^{k}$, and $0\ne z \in \Z$. Then
\begin{equation} \label{eq5534.520}
\frac{1}{\left |\sin(\pi^{k+1}z)\right |}\ll \frac{1}{\left |\sin\left (\pi^{k+1}q_n\right )\right |}.
\end{equation} 
\end{lem}
\begin{proof}[\textbf{Proof}] By the best approximation principle, see Lemma \ref{lem2000.07}, 
\begin{equation} \label{eq5534.522}
\left | m-\pi^{k} z\right |\geq \left | p_n-\pi^{k} q_n\right |
\end{equation} 
for any integer $z \leq q_n$. Hence, 
\begin{eqnarray} \label{eq5534.573}
\frac{1}{\left | \sin\left ( \pi^{k+1} z\right) \right |}
&=&\frac{1}{\left |\sin\left ( \pi m-\pi^{k+1} z\right)\right |} \\
&\leq&\frac{1}{\left |\sin\left ( \pi p_n-\pi^{k+1} q_n\right)\right |} \nonumber\\
&=& \frac{1}{\left |\sin\left ( \pi^{k+1} q_n \right)\right | }\nonumber,
\end{eqnarray}
as $n \to \infty$. 
\end{proof}

\section{Upper Bound For $\left | 1/\sin \pi^{k+1} z \right |$} \label{s6699}
As shown in Lemma \ref{lem5534.505}, to estimate the upper bound of the function $1/|\sin \pi^{k+1} z|$ over the real numbers $z \in\R$, it is sufficient to fix $z=q_n$, and select a real number $x \in \R$ such that $q_n \asymp x$. This idea is demonstrated below for small integer parameter $k\geq 1$.

\begin{lem}\label{lem6699.705} Let $k\geq 1$ be a small fixed integer, let $\{p_n/q_n : n \geq 1\}$ be the sequence of convergents of the real number $\pi^{k}$, and define the associated sequence
\begin{equation} \label{eq6699.702} 
x_n=\left (\frac{2^{2+2v_2}+1}{2^{2+2v_2}}\right )\frac{q_n}{\pi^{k}},
\end{equation} 
where $v_2=v_2(q_n)=\max\{v:2^v\mid q_n\}$ is the $2$-adic valuation, and $n \geq 1$. Then
\begin{enumerate} [font=\normalfont, label=(\roman*)]
\item$\displaystyle  \sin\left (2 (x_n-1/2)+1)\pi^{k+1}q_n\right ) =\pm 1$.
\item$\displaystyle  \sin\left (2 (x_n+1/2)+1)\pi^{k+1}q_n\right ) =\pm \cos 2\pi^{k+1}q_n$.
\item$\displaystyle  \left | \sin\left (2 x_n+1/2)\pi^{k+1}q_n\right ) \right |= 1 +O\left (\frac{1}{q_n^2}\right ),$  \tabto{12cm} as $n \to \infty$.
\end{enumerate}
\end{lem}
\begin{proof}[\textbf{Proof}] Observe that the value $x_n$ in \eqref{eq6699.702} yields 
\begin{equation} \label{eq6699.720}
\sin(2\pi^{k+1}q_nx_n)=\sin\left (2\pi^{k+1}q_n\left (\frac{2^{2+2v_2}+1}{2^{2+2v_2}}\right )\frac{q_n}{\pi^{k}}\right )=\sin\left (\frac{\pi}{2}\cdot w_n\right )=\pm1,
\end{equation}
and 
\begin{equation} \label{eq6699.724}
\cos\left (2\pi^{k+1}q_nx_n\right )=\cos\left (2\pi^{k+1}q_n\left (\frac{2^{2+2v_2}+1}{2^{2+2v_2}}\right )\frac{q_n}{\pi^{k}}\right )=\cos\left (\frac{\pi}{2}\cdot w_n\right )=0,
\end{equation}
where
\begin{equation} \label{eq6699.722}
w_n=\left (\frac{2^{2+2v_2}+1}{2^{2v_2}}\right ) q_n^{2} 
\end{equation}
is an odd integer.
(i) Routine calculations yield this:
\begin{eqnarray} \label{eq6699.703}
\sin((2(x_n-1/2)+1) \pi^{k+1}q_n)
&=&\sin\left (2\pi^{k+1}q_nx_n\right ) \\
&=&\sin\left (2\pi^{k+1}q_n\left (\frac{2^{2+2v_2}+1}{2^{2+2v_2}}\right )\frac{q_n}{\pi^{k}}\right )  \nonumber\\
&=&\sin\left (\frac{\pi}{2}\cdot w_n\right )  \nonumber\\
&=&\pm1 \nonumber,
\end{eqnarray}
(ii) Routine calculations yield this:
\begin{eqnarray} \label{eq6699.713}
\sin\left ((2(x_n+1/2) +1) \pi^{k+1}q_n \right )
&=&\sin(2\pi^{k+1}q_nx_n+ 2\pi^{k+1}q_n) \\
&=&\sin(2\pi^{k+1}q_nx_n)\cos( 2\pi^{k+1}q_n) \nonumber\\
&&\qquad \qquad + \cos(2\pi^{k+1}q_nx_n)\sin( 2\pi^{k+1}q_n)  \nonumber.
\end{eqnarray}
Substituting \eqref{eq6699.720} and \eqref{eq6699.724} into \eqref{eq6699.703} return 
\begin{equation} \label{eq6699.735}
\sin\left (2(x_n+1/2)+1)\pi^{k+1}q_n\right )=\pm\cos\left (2 \pi^{k+1}q_n\right).
\end{equation}
(iii) This follows from the previous result:
\begin{eqnarray} \label{eq6699.727}
\left |\sin\left (2(x_n+1/2)+1)\pi^{k+1}q_n\right ) \right |
&=&\left |\pm\cos\left (2 \pi^{k+1}q_n\right) \right |\\
&=&\left |\pm\cos\left ( 2\pi p_n-2\pi^{k+1} q_n\right)\right | \nonumber\\
&=&\left |\pm\cos\left ( 2\pi \left 
(p_n-\pi^{k} q_n\right ) \right)\right | \nonumber\\
&=  &1+O\left (\frac{1}{q_n^2}\right )\nonumber,
\end{eqnarray}
since the sequence of convergents satisfies $\left | p_n-\pi^{k}q_n \right | \leq 1/q_n$ as $n \to \infty$. 
\end{proof}
\begin{lem}\label{lem6699.805} Let $k\geq 1$ be a small fixed integer, let $\{p_n/q_n : n \geq 1\}$ be the sequence of convergents of the real number $\pi^{k}$, and define the associated sequence
\begin{equation} \label{eq6699.802} 
x_n=\left (\frac{2^{2+2v_2}+1}{2^{2+2v_2}}\right )\frac{q_n}{\pi^{k}}  ,
\end{equation} 
where $v_2=v_2(q_n)=\max\{v:2^v\mid q_n\}$ is the $2$-adic valuation, and $n \geq 1$. Then
\begin{equation} \label{eq6699.725}
\left |\sin\left (2x^{*}+1)\pi^{k+1}q_n\right ) \right |\asymp 1,
\end{equation}
where $x^{*}\in [x_n-1/2, x_n+1/2]$ is an integer.
\end{lem}

\begin{proof}[\textbf{Proof}] Consider the continuous function $f(x)=\left |\sin\left (2x+1)\pi^{k+1}q_n\right ) \right |$ over the interval $[x_n-1/2, x_n+1/2]$. By Lemma \ref{lem6699.705}, it has a local maximal at $x=x_n-1/2 \in \R$: 
\begin{eqnarray} \label{eq6699.303}
\left | \sin\left ((2x+1)\pi^{k+1}z\right )\right | &= &\left |\sin\left ((2(x_n-1/2)+1) \pi^{k+1} q_n\right )\right |\\
&=&1 \nonumber,
\end{eqnarray}
and it has a local minimal at $x=x_n+1/2 \in \R$:
\begin{eqnarray} \label{eq6699.303}
\left | \sin\left ((2x+1)\pi^{k+1}z\right )\right | &= &\left |\sin\left ((2(x_n+1/2)+1) \pi^{k+1} q_n\right )\right |\\
&= &1+O\left (\frac{1}{q_n^2}\right ) \nonumber.
\end{eqnarray}
Since $f(x)$ is continuous over the interval $[x_n-1/2, x_n+1/2]$, it follows that  
\begin{equation} \label{eq6699.877}
1 +O\left (\frac{1}{q_n^2}\right ) \leq \left | \sin\left ((2x^{*}+1)\pi^{k+1}z\right )\right | \leq 1 \nonumber
\end{equation}
for any integer $x^{*}\in [x_n-1/2, x_n+1/2]$ \end{proof}

\begin{thm} \label{thm6699.300}  If $k\geq 1$ is a small fixed integer, and $z \in \N$ is a large integer, then,
\begin{equation} \label{eq6699.300}
\left |\frac{1}{\sin \pi^{k+1} z}\right |\ll \left |z\right |.
\end{equation}
\end{thm}
\begin{proof}[\textbf{Proof}] Let $\{p_n/q_n : n \geq 1\}$ be the sequence of convergents of the real number $\pi^{k}$. Since the denominators sequence $\{q_n : n \geq 1\}$ maximize the reciprocal sine function $1/\sin \pi^{k+1} z$, see Lemma \ref{lem5534.505}, it is sufficient to prove it for $z=q_n$. Define the associated sequence
\begin{equation} \label{eq6699.302} 
x_n=\left (\frac{2^{2+2v_2}+1}{2^{2+2v_2}}\right )\frac{q_n}{\pi^{k}}  ,
\end{equation} 
where $v_2=v_2(q_n)=\max\{v:2^v\mid q_n\}$ is the $2$-adic valuation, and $n \geq 1$. Let $f(x)=\left | \sin\left ((2x+1)\pi^{k+1}z\right )\right | $, and let $z= q_n$. The function $f(x)$ is bounded over the interval $[x_n-1/2,x_n+1/2]$, see Lemma \ref{lem6699.705}.  Replacing the integer parameters $x^{*}\in [x_n-1/2, x_n+1/2]$, $z= q_n$, and applying Lemma \ref{lem6699.705} return 
\begin{eqnarray} \label{eq6699.303}
\left | \sin\left ((2x+1)\pi^{k+1}z\right )\right | &= &\left |\sin\left ((2x^{*}+1) \pi^{k+1} q_n\right )\right |\\
&\asymp&1 \nonumber.
\end{eqnarray}
Rewrite the reciprocal sine function in terms of the harmonic kernel in Definition \ref{dfn5534.100}, and splice all these information together, to obtain
\begin{eqnarray} \label{eq6699.313}
\left |\frac{1}{\sin \pi^{k+1}z}\right |  &=& \left |\frac{\mathcal{D}_{x}(\pi^{k+1}z)}{\sin((2x+1) \pi^{k+1}z)}\right |\nonumber\\
&\ll&\left |\mathcal{D}_{x^{*}}  \right | \left |\frac{1}{\sin((2x^{*}+1) \pi^{k+1}q_n)}\right |\\
&\ll&\left |x^{*}\right |\cdot 1\nonumber\\
&\ll&\left |z\right |\nonumber
\end{eqnarray}
since $|z|\asymp x^{*}\asymp p_n\asymp q_n$, and the trivial estimate $\left |\mathcal{D}_x(z) \right | \ll \left |x\right |$.
\end{proof}


\section{Elementary Techniques} \label{s5599}
A similar case arises for the sine function $\sin(\alpha \pi n$ as $n \to \infty$. This is handles by the observing that $1/q_n^{a-1}< \left | p_n-\alpha q_n\right | \leq 1/q_n$ is small as $n \to \infty$. Thus, the denominators sequence $\{q_n : n \geq 1\}$ maximize the reciprocal sine function $1/\sin z$. Hence, it is sufficient to consider the infinite series over the denominators sequence $\{q_n: n \geq 1\}$.
\begin{thm}  \label{thm5599.45}  Let $\alpha \in \R$ be an irrational number of irrationality measure $\mu(\alpha)=a$. Then 
\begin{equation} \label{eq5599.197}
\frac{1}{\left |\sin \alpha \pi q_n \right | }\ll q_n^{a-1}.
\end{equation}
\end{thm}
\begin{proof}[\textbf{Proof}] Let $\alpha=[a_0,a_1,a_2, \ldots ]$ be the continued fraction of the number $\alpha$, and let $\{p_n/q_n : n \geq 1\}$ be the sequence of convergents. Then 
\begin{eqnarray} \label{eq5599.199}
\frac{1}{\left |\sin \alpha \pi q_n \right | }&=&\frac{1}{\left |\sin( \pi p_n-\alpha \pi q_n) \right | }\\
&\geq &\frac{1}{ \pi \left | p_n-\alpha q_n \right |} \nonumber\\
&>& q_n^{a-1}\nonumber.
\end{eqnarray}
as $n \to \infty$. 
\end{proof}

\section{Asymptotic Expansions Of The Sine Function} \label{s938}

\begin{thm} \label{thm938.40}   Let $p_n/q_n$ be the sequence of convergents of the irrational number $\pi$. Then, 
\begin{equation} \label{eq938.400}
 \sin \left ( \frac{1}{p_n}\right )  \asymp \sin  p_n 
\end{equation}
as $p_n \to \infty$.
\end{thm}

\begin{proof}[\textbf{Proof}] The Taylor series of the sine function leads to the inequality 
	\begin{equation}\label{eq938.410}
		z-\frac{z^3}{6}\leq \sin z\leq z,    
	\end{equation}
	where
	$|z|<1$. Replacing $z=\left |\pi q_n-p_n \right |$ 
	 yields the symmetric inequality
	\begin{equation} \label{eq938.420}
		\left |\pi q_n-p_n \right |\ll \left | \sin (\pi q_n-p_n )\right | \ll \left |\pi q_n-p_n \right |.
	\end{equation}
Moreover, $\left |\pi q_n-p_n \right |\leq 1/q_n$, implies the
the associated symmetric inequality
\begin{equation} \label{eq938.430}
	\left |\pi q_n-p_n \right |\ll \frac{1}{q_n}\ll \left | \sin (\pi q_n-p_n )\right | \ll \left |\pi q_n-p_n \right |\ll \frac{1}{q_n}.
\end{equation}
Since $p_n\asymp q_n$, the last inequality leads to the relation 
\begin{eqnarray} \label{eq938.430}
\left | \sin p_n \right |&=&\left |\pi q_n-p_n \right |\\
&\gg &\frac{1}{p_{n}} \nonumber\\
&\gg  & \sin \left (\frac{1}{p_{n}}\right )\nonumber.
\end{eqnarray}
On the other direction,
\begin{eqnarray} \label{eq938.82}
\sin \left (\frac{1}{p_{n}}\right )  &\gg&  \frac{1}{p_{n}} \\
&\gg&  \frac{1}{q_{n}} \nonumber\\
&\gg& \left | p_{n}-\pi q_{n}\right |  \nonumber\\
&= &\left |\sin\left ( p_{n}-\pi q_{n}\right )\right |
\nonumber\\
&=  &\left | \sin \left (p_{n}\right ) \right |\nonumber,
\end{eqnarray}
These prove that 
\begin{equation} \label{eq938.77}
\sin p_n \gg  \sin \left (\frac{1}{p_{n}}\right )  \qquad \text{ and } \qquad \sin \left (\frac{1}{p_{n}}\right )  \gg \sin p_n,
\end{equation}
as $n \to \infty$.

\end{proof}

\begin{lem} \label{lem938.43} Let $p_n/q_n$ be the sequence of convergents of the irrational number $\pi=[a_0;a_1,a_2, \ldots ]$. Then, the followings hold.
\begin{enumerate} [font=\normalfont, label=(\roman*)]
\item  $\displaystyle \sin p_n  = \sin \left ( p_n-\pi q_n\right )$, \tabto{10cm}for all $p_n$ and $q_n$.
\item  $\displaystyle \sin \left ( \frac{1}{p_n}\right )  = \frac{1}{p_n}\left (1-\frac{1}{3!}\frac{1}{p_n^2}+\frac{1}{5!}\frac{1}{p_n^4}-\cdots \right )$, \tabto{10cm}as $p_n \to \infty$.
\item  $\displaystyle \cos 2p_n  = \cos \left ( 2(p_n-\pi q_n)\right )$, \tabto{10cm}for all $p_n$ and $q_n$.
\end{enumerate}
\end{lem}
\section{Bounded Partial Quotients} \label{s590}
A result for the partial quotients based on the properties of continued fractions and the asymptotic expansions of the sine function is derived below.

\begin{thm} \label{thm590.45}   Let $\pi=[a_0;a_1,a_2, \ldots ]$ be the continued fraction of the real number $\pi \in \R$. Then, the $n$th partial quotients $a_n\in \N$ are bounded. Specifically,  $a_n=O(1)$ for all $n \geq 1$.
\end{thm}
\begin{proof}[\textbf{Proof}] Consider the real numbers
\begin{equation} \label{eq590.15}
C=p_{n+1}-\pi q_{n+1} -\frac{1}{q_{n}} \quad \text{ and } \quad D=\frac{1}{q_n} .
\end{equation}
 For all sufficiently large integers $n \geq 1$, the Dirichlet approximation theorem and the addition formula $\sin(C+D)=\cos C \sin D+\cos D \sin C$ lead to 
\begin{eqnarray} \label{eq590.22}
 \frac{1}{q_{n+1}}&\gg& \left | p_{n+1}-\pi q_{n+1}  \right |\\
&\gg& \left | \sin\left (p_{n+1}-\pi q_{n+1}\right )  \right |\nonumber\\
&=& \left |\sin\left ( p_{n+1}-\pi q_{n+1}- \frac{1}{q_n}+\frac{1}{q_n}\right )\right | \nonumber\\
&=&  \left |\cos C \sin \left (\frac{1}{q_{n}}\right )+\cos \left (\frac{1}{q_{n}}\right ) \sin C  \right |\nonumber.
\end{eqnarray}
The sequence $\{p_n: n \geq 1\}$ maximizes the cosine function and minimizes the sine function. Hence, by Lemma \ref{lem938.99},
\begin{equation} \label{eq590.77}
1-\frac{2}{q_{n}^2}\leq  1-\frac{C^2}{2} \leq\cos C \leq 1 ,
\end{equation}
and 
\begin{equation} \label{eq590.75}
 \frac{1}{2q_{n}}-\frac{1}{48q_{n}^3}\leq  C-\frac{C^3}{6}\leq \sin C \leq \frac{2}{q_{n}}.
\end{equation}
Substituting these estimates into the reverse triangle inequality $|X+Y| \geq ||X|-|Y||$, produces
\begin{eqnarray} \label{eq590.79}
 \frac{1}{q_{n+1}}&\gg& 
\left |\cos C \sin \left (\frac{1}{q_{n}}\right )+\cos \left (\frac{1}{q_{n}}\right ) \sin C  \right |\\
&\geq & \left |\left | \left ( 1 \right )\left (\frac{1}{q_{n}}-\frac{1}{6q_{n}^3}\right )\right |  -\left |\left (1\right ) \left ( \frac{1}{2q_{n}}-\frac{1}{48q_{n}^3}\right )\right | \right |
\nonumber\\
&\geq & \frac{c_0}{ q_{n}}\nonumber,
\end{eqnarray}
where $c_0>0$ is a constant. These show that
\begin{equation} \label{eq590.24}
\frac{1}{q_{n+1}}\gg\frac{1}{q_{n}} 
\end{equation}
as $n \to \infty$. Furthermore, \eqref{eq590.24} implies that $q_n\gg q_{n+1}=a_{n+1}q_n+q_{n-1}$ as claimed.
\end{proof}

\begin{lem} \label{lem938.95}   Let $p_n/q_n$ be the sequence of convergents of the real number $\pi=[a_0;a_1,a_2, \ldots ]$. Then, as $n \to \infty$, 
\begin{multicols}{2}
\begin{enumerate} [font=\normalfont, label=(\roman*)]
\item  $\displaystyle \frac{1}{2q_n} \leq \left | p_{n+1}-\pi q_{n+1}- \frac{1}{q_n}\right | \leq \frac{2}{q_n}$.
\item  $\displaystyle \frac{1}{2p_n} \leq \left | p_{n+1}-\pi q_{n+1}- \frac{1}{q_n}\right | \leq \frac{13}{p_n}$.
\end{enumerate}
\end{multicols}
\end{lem}

\begin{proof}[\textbf{Proof}] (i) Using the triangle inequality and Dirichlet approximation theorem yield the upper bound
\begin{eqnarray} \label{eq938.90}
\left | p_{n+1}-\pi q_{n+1}- \frac{1}{q_n}\right | &\leq& \left | p_{n+1}-\pi q_{n+1} \right |+\frac{1}{q_n} \\
&\leq& \frac{c_0}{q_{n+1}} +\frac{1}{q_n} \nonumber\\
&\leq& \frac{2}{q_{n}} \nonumber,
\end{eqnarray}
since $q_{n+1} =a_{n+1}q_n+q_{n-1}$. The lower bound
\begin{eqnarray} \label{eq938.96}
\left |  \frac{1}{q_n}-p_{n+1}-\pi q_{n+1}\right | &\geq& \left |\frac{1}{q_n}-\left | p_{n+1}-\pi q_{n+1} \right |  \right |\\
&\geq& \frac{1}{q_n}-\frac{c_1}{q_{n+1}}  \nonumber\\
&\geq& \frac{1}{2q_{n}} \nonumber,
\end{eqnarray}
where  $c_0>0$ and $c_1>0$ are small constants. (ii) Use the Hurwitz approximation theorem 
\begin{equation} \label{eq938.92}
\pi q_n-\frac{1}{\sqrt{5}q_{n}^2}\leq p_{n}\leq \pi q_n +\frac{1}{\sqrt{5}q_{n}^2} 
\end{equation}
to convert it to the upper bound in term of $p_n$: 
\begin{equation} \label{eq938.94}
\left | p_{n+1}-\pi q_{n+1}- \frac{1}{q_n}\right | \leq \frac{2}{q_{n}} \leq \frac{13}{p_{n}} \nonumber.
\end{equation}
 as $n \to \infty$ as claimed.
\end{proof}

\begin{lem} \label{lem938.99}   Let $\{p_n/q_n:n \geq 1\}$ be the sequence of convergents of $\pi$, and let $C=p_{n+1}-\pi q_{n+1} -q_n$. Then,  
\begin{enumerate} [font=\normalfont, label=(\roman*)]
\item  $\displaystyle 1-\frac{2}{q_{n}^2}\leq\cos C \leq 1 -\frac{2}{q_{n}^2}+\frac{1}{6q_{n}^4}$.
\item   $\displaystyle \frac{1}{2q_{n}}-\frac{1}{48q_{n}^3}\leq \sin C \leq \frac{2}{q_{n}}$.
\end{enumerate}
\end{lem}
\begin{proof}[\textbf{Proof}] Use Lemma \ref{lem938.95}.
\end{proof}
\section{Statistics And Example} \label{s3547}
The result of Theorem \ref{thm590.45} indicates that the real number $\pi$ fits the profile of a random irrational number. The geometric mean value of the partial quotients of a random irrational number has the value
\begin{equation} \label{eq3547.26}
K_0=\lim_{n \to \infty} \left (\prod_{k \leq n} a_k\right )^{1/n}=2.6854520010\ldots, 
\end{equation}
which is known as the Khinchin constant. In addition, the Gauss-Kuzmin distribution specifies the frequency of each value by
\begin{equation} \label{eq3547.28}
p(a_k=k)= -\log_2 \left (1-\frac{1}{(k+1)^2}\right ). 
\end{equation}
For a random irrational number, the proportion of partial quotients $a_k=1,2$ is about $p(a_1)+p(a_2)\leq .5897$. A large value $a_k\geq 10^6$ is very rare. In fact,
\begin{equation} \label{eq3547.28}
p(a_k=10^6)= -\log_2 \left (1-\frac{1}{(10^6+1)^2}\right )=0.00000000000144, 
\end{equation}
and the sum of probabilities is
\begin{equation} \label{eq3547.29}
\sum_{k \geq 10^6}p(a_k=k)= \sum_{k \geq 10^6}-\log_2 \left (1-\frac{1}{(k+1)^2}\right )\leq 0.000001. 
\end{equation}

\begin{exa} \label{exa3547.10} {\normalfont  A short survey of the partial quotients $\pi=[a_0;a_1,a_2, \ldots ]$ is provided here to sample this phenomenon, most computer algebra system can generate a few  thousand terms within minutes. This limited numerical experiment established that $a_n\leq 21000$ for $n\leq 10000$. The value $a_{432}=20776$ is the only unusual partial quotient. This seems to be an instance of the Strong Law of Small Number.
\begin{equation}
\begin{split}
\pi=[3, 7, 15, 1, 292, 1, 1, 1, 2, 1, 3, 1, 14, 2, 1, 1, 2, 2, 2, 2, 1, 84, 
2, 1, 1, 15, 3, 13, 1, 4, 2, \\
6, 6, 99, 1, 2, 2, 6, 3, 5, 1, 1, 6, 8, 
1, 7, 1, 2, 3, 7, 1, 2, 1, 1, 12, 1, 1, 1, 3, 1, 1, 8, 1, 1, \ldots].
\end{split}
\end{equation}  
}\end{exa}

\section{Convergence Of The Flint Hills Series I} \label{s5444}
The first analysis of the the convergence of the Flint Hills series is based on the upper bound of the reciprocal sine function $1/\sin n$ as $n \to \infty$, in Theorem \ref{thm6699.300}. This is equivalent to the upper bound of the harmonic summation kernel
\begin{equation} \label{eq5444.111}
\mathcal{D}_x(z)=\sum_{-x\leq n \leq x} e^{i 2nz}=\frac{\sin((2x+1) z)}{ \sin \left (  z \right )},
\end{equation}
consult Definition \ref{dfn5534.100}. 
Let $\pi=[a_0,a_1,a_2, \ldots ]$ be the continued fraction of the number $\pi$, and let $\{p_n/q_n : n \geq 1\}$ be the sequence of convergents. The difficulty in proving the convergence of the series \eqref{eq5444.96} arises from the sporadic maximal values of the function 
\begin{eqnarray} \label{eq5444.41}
\frac{1}{\left |\sin p_n \right | }&=&\frac{1}{\left |\sin( p_n-\pi q_n) \right | }\\
&\geq &\frac{1}{ \left | p_n-\pi q_n \right |} \nonumber\\
&>& q_n^{42}\nonumber.
\end{eqnarray}
at the integers $z=p_n$ as $n \to \infty$. The irrationality exponent 42 in the Diophantine inequality
\begin{equation} \label{eq5444.87}
 \frac{1}{q_n^{42}}\leq \left | \pi -\frac{p_n}{q_{n}}\right | \leq \frac{1}{q_n}
\end{equation}
was determined in \cite{MK53}.  Earlier questions on the convergence the series \eqref{eq5444.96} appears in \cite[p.\ 59]{PC07}, \cite[p.\ 583]{TG17}, \cite{AM11}, et alii.
\begin{thm} \label{thm5444.91} If the real numbers $u>0$ and $v>0$ satisfy the relation $u-v>0$, then the Flint Hills series 
\begin{equation} \label{eq5444.96}
\sum_{n \geq 1} \frac{1}{n^{u}\sin^v n} 
\end{equation}
is absolutely convergent. 
\end{thm}

\begin{proof}[\textbf{Proof}] As $1/q_n^{41}< \left | p_n-\pi q_n\right | \leq 1/q_n$ is small as $n \to \infty$, the numerators sequence $\{p_n : n \geq 1\}$ maximize the reciprocal sine function $1/\sin z$. Hence, it is sufficient to consider the infinite series over the numerators sequence $\{p_n: n \geq 1\}$. To accomplish this, write the series as a sum of a convergent infinite series and the lacunary infinite series (over the sequence of numerators):
\begin{eqnarray} \label{eq5444.51}
\sum_{n \geq 1} \frac{1}{n^{u}\sin^v n} 
&=&\sum_{\substack{m \geq 1\\ m\ne p_n}} \frac{1}{m^{u}\sin^v m}+\sum_{n  \geq1 } \frac{1}{p_n^{u}\sin \left (p_n \right )^v}\\
&\leq&c_0\sum_{n  \geq1 } \frac{1}{p_n^{u}\sin \left (p_n \right )^v} \nonumber.
\end{eqnarray}
where $c_0, c_1, c_2, c_3, c_4>0$ are constants. By Theorem \ref{thm6699.300}, the reciprocal of the sine function is bounded
\begin{equation} \label{eq5444.57}
\left | \frac{1}{\sin (p_n)} \right | 
\leq c_1 p_n.
\end{equation}
Applying this bound yields
 \begin{eqnarray} \label{eq5444.45}
\sum_{n  \geq1 } \frac{1}{p_n^{u}(\sin p_n)^v}
&\leq & \sum_{n \geq 1} \frac{\left( c_1p_n\right )^v}{p_n^{u}} \nonumber\\
&\leq&c_2\sum_{n \geq 1} \frac{1  }{p_n^{u-v}} .
\end{eqnarray}
By the Binet formula for quadratic recurrent sequences, the sequence $p_n=a_np_{n-1}+p_{n-2}$, $a_n \geq 1$,  has exponential growth, namely,
\begin{equation} \label{eq5444.86}
p_n\geq\frac{1}{\sqrt{5} }\left ( \frac{1+\sqrt{5}}{2 }\right )^n 
\end{equation}
for $n \geq 2$. Now, replacing \eqref{eq5444.86} into \eqref{eq5444.45} returns
\begin{eqnarray} \label{eq5444.47}
c_2\sum_{n \geq 1} \frac{1  }{p_n^{u-v}}
&\leq& c_3\sum_{n \geq 1} \left (\frac{2  }{1+\sqrt{5}}\right  )^{(u-v)n} \nonumber\\
&\leq& c_4 \sum_{n \geq 1} \left (\frac{1  }{2}\right  )^{(u-v)n}.
\end{eqnarray}
Hence, it immediately follows that the infinite series converges whenever $u-v>0$.
\end{proof}

\begin{exa} {\normalfont The infinite series 
\begin{equation} \label{eq5444.211}
\sum_{n \geq 1} \frac{1}{n^{3}\sin^2 n} 
\end{equation}
has $u-v=3-2=1>0$. Hence, by Theorem \ref{thm5444.91}, it is convergent.     
} 
\end{exa}

\begin{exa} {\normalfont Let $\varepsilon>0$ be an arbitrary small number. The infinite series 
\begin{equation} \label{eq5444.213}
\sum_{n \geq 1} \frac{1}{n^{1+\varepsilon}\sin n} 
\end{equation}
has $u-v=1+\varepsilon-1=\varepsilon>0$. Hence, by Theorem \ref{thm5444.91}, it is absolutely convergent.     
} 
\end{exa}

\begin{tikzpicture}
\begin{axis}[
    title={The Partial Sum $P_x=\sum_{n \leq x}\frac{1}{n^3 \sin^2 n}$},
    xlabel={$x$},
    ylabel={$P_x$},
    ymin=0, ymax=40,
    xtick={0,1,2, 3,21,22,354,355,400,500,600,700},
    ytick={0,5,10,15,20,25,30,35},
    legend pos=north west,
    ymajorgrids=true,
    grid style=dashed,
    width=.90\textwidth,
       height=0.5\textwidth,
]
 
\addplot[
    color=blue,
    mark=square,
    ]
    coordinates {
    (1, 1.41228293)(3,  3.42323343)(21, 3.555394)(22, 4.754112)(354, 4.807444)(355,  29.405625)(400, 29.405918)(500, 29.405964)
    };
 
\end{axis}
\end{tikzpicture}

\section{The Flint Hills Series And The Lacunary Sine Series}\label{s5001}
A comparison of the partial sum of the Flint Hills series 
\begin{equation} \label{eq5001.41}
P_x=\sum_{n \leq x} \frac{1}{n^{u}\sin^v n} 
\end{equation}
and the partial sum of the Lacunary Sine series 
\begin{equation} \label{eq5001.44}
Q_x=\sum_{p_n  \leq x } \frac{1}{p_n^{u}\sin \left (p_n \right )^v}
\end{equation}
respectively, is tabulated in Table \ref{t5001}. It demonstrates that the Lacunary Sine series infuses an overwhelming contribution to the complete sum.

\begin{table}
\centering
\caption{Comparison Of The Partial Sums $P_x$ and $Q_x$ For $(u,v)=(3,2)$.} \label{t5001}
\begin{tabular}{l|l|l}
$x$&$P_x$ &$Q_x$\\
\hline
 $1$ & 1.422829&1.422829\\
 $3$ & 3.423233&1.887049\\
 $22$ & 4.754112&3.085767\\
$355$ & 29.405625&27.683949\\
\end{tabular}
\end{table}

\section{Convergence Of The Flint Hills Series II} \label{s544}
 The second analysis of the the convergence of the Flint Hills series is based on the asymptotic relation
\begin{equation} \label{eq534.111}
\sin \left ( p_n\right )  \asymp  \sin \left ( \frac{1}{p_n}\right ),
\end{equation}
where $p_n/q_n \in \Q$ is the sequence of convergents of the real number $\pi \in \R$. In some way, the representation \eqref{eq534.111} removes any reference to the difficult problem of estimating the maximal value of the function $1/\sin n$ as $n \to \infty$. Earlier study of the convergence the series \eqref{eq544.96} appears in \cite[p.\ 583]{TG17}, \cite{AM11}, et alii.  
\begin{thm} \label{thm597.91} If the real numbers $u>0$ and $v>0$ satisfy the relation $u-v>0$, then the Flint Hills series 
\begin{equation} \label{eq544.96}
\sum_{n \geq 1} \frac{1}{n^{u}\sin^v n} 
\end{equation}
is absolutely convergent. 
\end{thm}

\begin{proof}[\textbf{Proof}]  Let $\{p_n/q_n : n \geq 1\}$ be the sequence of convergents of the real number $\pi$. Since the numerators sequence $\{p_n : n \geq 1\}$ maximize the reciprocal sine function $1/\sin z$. Hence, it is sufficient to consider the infinite series over the lacunary numerators sequence $\{p_n: n \geq 1\}$. Substituting the numerators sequence returns
\begin{eqnarray} \label{eq544.41}
\sum_{n \geq 1} \frac{1}{n^{u}\sin^v n} 
&\ll&\sum_{n  \geq1 } \frac{1}{p_n^{u}(\sin p_n)^v} \\
&\ll&\sum_{n  \geq1 } \frac{1}{p_n^{u}\sin \left (\frac{1}{p_n} \right )^v}\nonumber,
\end{eqnarray}
see Theorem \ref{thm938.40}. Substituting the Taylor series at infinity, see Lemma \ref{lem938.43}, return
\begin{eqnarray} \label{eq544.45}
\sum_{n  \geq1 } \frac{1}{p_n^{u}\sin \left (\frac{1}{p_n} \right )^v} 
&=& \sum_{n  \geq1 } \frac{1}{p_n^{u}\displaystyle \left ( \frac{1}{p_n}\left (1-\frac{1}{3!}\frac{1}{p_n^2}+\frac{1}{5!}\frac{1}{p_n^4}-\cdots \right )\right ) ^v} \nonumber\\
&\ll&\sum_{n \geq 1} \frac{1  }{p_n^{u-v}} .
\end{eqnarray}
By the Binet formula for recurrent quadratic sequences, the sequence $p_n=a_np_{n-1}+p_{n-2}$, $a_n \geq 1$,  has exponential growth, namely,
\begin{equation} \label{eq544.86}
p_n\geq\frac{1}{\sqrt{5} }\left ( \frac{1+\sqrt{5}}{2 }\right )^n 
\end{equation}
for $n \geq 2$. Replacing \eqref{eq544.86} into \eqref{eq544.45} returns
\begin{eqnarray} \label{eq544.47}
\sum_{n \geq 1} \frac{1  }{p_n^{u-v}}
&\ll& \sum_{n \geq 1} \left (\frac{2  }{1+\sqrt{5}}\right  )^{(u-v)n} \nonumber\\
&\ll& \sum_{n \geq 1} \left (\frac{1  }{2}\right  )^{(u-v)n} .
\end{eqnarray}
Hence, it immediately follows that the infinite series converges whenever $u-v>0$.
\end{proof}

\section{Convergence Of Some Flint Hills Type Series} \label{s7444}
The analysis for the convergence of some modified Flint Hills series as
\begin{equation} \label{eq7444.15}
\sum_{n \geq 1} \frac{1}{n^{u}\sin^v \alpha \pi n}
\end{equation}
have many similarities to the of the Flint Hills series.
Given an irrational number $\alpha \in \R$ of irrationality measure $\mu(\alpha)=a$, let $\pi=[a_0,a_1,a_2, \ldots ]$ be the continued fraction of the number $\alpha$, and let $\{p_n/q_n : n \geq 1\}$ be the sequence of convergents. The difficulty in proving the convergence of the series \eqref{eq5444.96} arises from the sporadic maximal values of the function 
\begin{eqnarray} \label{eq7444.41}
\frac{1}{\left |\sin \alpha \pi q_n \right | }&=&\frac{1}{\left |\sin( \pi p_n-\alpha \pi q_n) \right | }\\
&\geq &\frac{1}{ \pi \left | p_n-\alpha q_n \right |} \nonumber\\
&>& q_n^{a-1}\nonumber.
\end{eqnarray}
as $n \to \infty$. 
\begin{thm} \label{thm7444.91} If the real numbers $u>0$ and $v>0$ satisfy the relation $u-(a-1)v>0$, then the series 
\begin{equation} \label{eq7444.96}
\sum_{n \geq 1} \frac{1}{n^{u}\sin^v \alpha \pi n} 
\end{equation}
is absolutely convergent. 
\end{thm}

\begin{proof}[\textbf{Proof}] As $1/q_n^{a-1}< \left | p_n-\alpha q_n\right | \leq 1/q_n$ is small as $n \to \infty$, the numerators sequence $\{p_n : n \geq 1\}$ maximize the reciprocal sine function $1/\sin z$. Hence, it is sufficient to consider the infinite series over the numerators sequence $\{p_n: n \geq 1\}$. To accomplish this, write the series as a sum of a convergent infinite series and the lacunary infinite series (over the sequence of numerators):
\begin{eqnarray} \label{eq7444.41}
\sum_{n \geq 1} \frac{1}{n^{u}\sin^v n} 
&=&\sum_{\substack{m \geq 1\\ m\ne p_n}} \frac{1}{m^{u}\sin^v m}+\sum_{n  \geq1 } \frac{1}{q_n^{u}\sin \left (\alpha \pi q_n \right )^v}\\
&\leq&c_0\sum_{n  \geq1 } \frac{1}{q_n^{u}\sin \left (\alpha \pi q_n \right )^v} \nonumber.
\end{eqnarray}
where $c_0, c_1, c_2, c_3, c_4>0$ are constants. By hypothesis, the reciprocal of the sine function is bounded by
\begin{equation} \label{eq5444.47}
\left | \frac{1}{\sin (\alpha \pi q_n)} \right | 
\leq c_1 q_n^{a-1},
\end{equation}
see Theorem \ref{thm5599.45} for more details. Applying this bound yields
 \begin{eqnarray} \label{eq7444.45}
\sum_{n  \geq1 } \frac{1}{q_n^{u}\sin \left (\alpha \pi q_n \right )^v}
&\leq & \sum_{n \geq 1} \frac{\left( c_1q_n\right )^{(a-1)v}}{q_n^{u}} \nonumber\\
&\leq&c_2\sum_{n \geq 1} \frac{1  }{q_n^{u-(a-1)v}} .
\end{eqnarray}
By the Binet formula for quadratic recurrent sequences, the sequence $q_n=a_nq_{n-1}+q_{n-2}$, $a_n \geq 1$,  has exponential growth, namely,
\begin{equation} \label{eq7444.86}
q_n\geq\frac{1}{\sqrt{5} }\left ( \frac{1+\sqrt{5}}{2 }\right )^n 
\end{equation}
for $n \geq 2$. Now, replacing \eqref{eq7444.86} into \eqref{eq5444.45} returns
\begin{eqnarray} \label{eq7444.47}
c_2\sum_{n \geq 1} \frac{1  }{q_n^{u-(a-1)v}}
&\leq& c_3\sum_{n \geq 1} \left (\frac{2  }{1+\sqrt{5}}\right  )^{(u-(a-1)v)n} \nonumber\\
&\leq& c_4\sum_{n \geq 1} \left (\frac{1  }{2}\right  )^{(u-(a-1)v)n} .
\end{eqnarray}
Hence, it immediately follows that the infinite series converges whenever $u-(a-1)v>0$.
\end{proof}

\begin{exa} {\normalfont The irrational number $\alpha=\sqrt{2}$ has irrationality measure $\mu(\sqrt{2})=2$, and the infinite series 
\begin{equation} \label{eq7444.211}
\sum_{n \geq 1} \frac{1}{n^{3}\sin^2 \sqrt{2} \pi n} 
\end{equation}
has $u-(a-1)v=3-2=1>0$. Hence, by Theorem \ref{thm7444.91}, it is absolutely convergent.     
} 
\end{exa}

\begin{exa} {\normalfont Let $\varepsilon>0$ be an arbitrary small number. The irrational number $\alpha=\sqrt[3]{2}$ has irrationality measure $\mu(\sqrt[3]{2})=2$, and the infinite series 
\begin{equation} \label{eq7444.213}
\sum_{n \geq 1} \frac{1}{n^{1+\varepsilon}\sin \sqrt[3]{2} \pi n} 
\end{equation}
has $u-(a-1)v=1+\varepsilon-1=\varepsilon>0$. Hence, by Theorem \ref{thm7444.91}, it is absolutely convergent.     
} 
\end{exa}

\section{Convergence Of The Flat Hills Series} \label{s4444}
Let $\{x\}$ be the fractional part function and let $|| \alpha ||= \min \{ |n-\alpha|: n \in \Z\}$ be the least distance to the nearest integer. Another classes of problems arise from the other properties of the number $\pi$ and other irrational numbers. One of these problems is the convergence of the \textit{Flat Hills series}.

\begin{dfn} {\normalfont Let $a>1$ and $b\ne 0$ be a pair of real parameters. A Flat Hills series is defined by an infinite sum of the following forms. 
\begin{multicols}{2}
 \begin{enumerate} [font=\normalfont, label=(\roman*)]
\item$ \displaystyle 
\sum_{n \geq 1} \frac{1}{n^{a}\sin^b || \pi^n ||}$,
\item$\displaystyle
 \sum_{n \geq 1} \frac{1}{n^{a}\sin^b || \pi 10^n ||}$,
\item$\displaystyle  \sum_{n \geq 1} \frac{1}{n^{a}\sin^b \{ \pi ^n \}} $,
\item$\displaystyle  \sum_{n \geq 1} \frac{1}{n^{a}\sin^b \{ \pi 10^n \}} $.
\end{enumerate}
\end{multicols}
}
\end{dfn}
The known properties of some algebraic irrational numbers $\alpha \in \R$ can be used to determine the convergence or divergence of the Flat Hills series
\begin{equation} \label{eq4444.01}
\sum_{n \geq 1} \frac{1}{n^{a}\sin^b || \alpha^n ||} ,
\end{equation}
see Exercise \ref{exe79000.01}. In the case $\alpha=\pi$, the analytic properties of sequences such as $\{|| \pi^n ||:n \geq 1\}$, and $\{|| \pi 10^n ||:n \geq 1\}$ are unknown. Accordingly, the convergence or divergence of the infinite series \eqref{eq4444.01} are unknown. Likewise, for the case $\alpha=e$, the analytic properties of sequences such as $\{|| e^n ||:n \geq 1\}$, and $\{|| e 10^n ||:n \geq 1\}$ are unknown. Accordingly, the convergence or divergence of the infinite series \eqref{eq4444.01} are unknown, see Lemma \ref{lem2200.15} for some details.

\section{Representations}
The Flint Hills series and similar class of infinite series have representations in terms of other analytic functions. These  reformulations could be useful in further analysis of these series.

\subsection{Gamma Function Representation}
The gamma function reflection formula produces a new reformulation as
\begin{equation} \label{eq534.115}
\sum_{n \geq 1} \frac{1}{n^{u}\sin^v n} =\sum_{n \geq 1} \frac{\Gamma(1-n/\pi)^v\Gamma(n/\pi)^v}{n^{u}}, 
\end{equation}
see Lemma \ref{lem514.43}. 

\subsection{Zeta Function Representation}
The zeta function reflection relation 
\begin{equation} \label{eq534.119}\zeta(s)=2^{s}\pi^{1-s}\Gamma(1-s) \sin \left ( \pi s/2 \right ) \zeta(1-s)\end{equation} 
evaluated at $s=2n/\pi$ produces a new reformulation as
\begin{equation} \label{eq534.115}
\sum_{n \geq 1} \frac{1}{n^{u}\sin^v n} =\sum_{n \geq 1} \frac{\left (2^{2n/\pi}\pi^{1-2n/\pi}\Gamma(1-2n/\pi) \zeta(1-2n/\pi)\right )^v}{n^{u}\zeta(2n/\pi)^v}, 
\end{equation}

\subsection{Harmonic Kernel Representation}
The harmonic summation kernel is defined by
\begin{equation} \label{eq534.200}
\mathcal{K}_x(z)=\sum_{-x\leq n \leq x} e^{i 2nz}=\frac{\sin((2x+1)z)}{ \sin \left ( z \right )},
\end{equation} 
where $x , z \in \C$ are complex numbers. For $z \ne k \pi$, the  harmonic summation kernel has the upper bound $\left |\mathcal{K}_x(z) \right | \ll |x|$. Replacing it into the series produces a new reformulation as
\begin{equation} \label{eq534.215}
\sum_{n \geq 1} \frac{1}{n^{u}\sin^v n} =\sum_{n \geq 1} \frac{\mathcal{K}_x(n)^v}{n^{u}\sin^v ((2x+1)n)}. 
\end{equation}
Here the choice of parameter $x =\alpha n$, where $\alpha >0$ is an algebraic irrational number, shifts the large value of the reciprocal sine function $1/\sin n$ to $\mathcal{K}_x(z)$, and the $1/\sin (2x+1)n=\sin  (2\alpha n+1)n$ remains bounded. Thus, it easier to prove the convergence of the series \eqref{eq534.215}.

\section{Beta and Gamma Functions} \label{s514}

Certain properties of the beta function $B(a,b)=\int_0^1 t^{a-1}(1-t)^{b-1}dt$, and gamma function $\Gamma(z)=\int_0^{\infty}t^{z-1}e^{-t} dt$ are useful in the proof of the main result.
\begin{lem} \label{lem514.43}   Let $B(a,b)$ and $\Gamma(z)$ be the beta function and gamma function of the complex numbers $a,b, z \in \C-\{0, -1, -2,-3, \ldots\}$. Then,
\begin{enumerate} [font=\normalfont, label=(\roman*)]
\item  $\displaystyle  \Gamma(z+1)=z\Gamma(z)$, the functional equation.
\item  $\displaystyle  B(a,b) =\frac{\Gamma(a)\Gamma(b)}{\Gamma(a+b)}$, the multiplication formula.
\item  $\displaystyle  \frac{B(1-z,z)}{\pi} =\frac{1}{\sin \pi z}$.
\item  $\displaystyle  \frac{\Gamma(1-z)\Gamma(z)}{\pi} =\frac{1}{\sin \pi z}$, the reflection formula.
\end{enumerate}
\end{lem}
\begin{proof}[\textbf{Proof}] Standard analytic methods on the $B(a,b)$, and gamma function $\Gamma(z)$, see \cite[Equation 5.12.1]{DLMF}, \cite[Chapter 1]{AR99}, et cetera.
\end{proof}

The current result on the irrationality measure $\mu(\pi)=7.6063$, see Table  \ref{t9001}, of the number $\pi$ implies that

\begin{equation} \label{eq514.22}
\left |\Gamma(z+1)\Gamma(z) \right | \ll |z|^{7.6063}.
\end{equation}
A  sharper upper bound is computed here.
\begin{lem} \label{lem514.25} Let $\Gamma(z)$ be the gamma function of a complex number $z \in \C$, but not a negative integer $z \ne 0,-1,-2, \ldots$. Then,
\begin{equation} \label{eq514.22}
\left |\Gamma(z+1)\Gamma(z) \right | \ll |z|.
\end{equation}
\end{lem}

\begin{proof}[\textbf{Proof}] Assume $\Re e(z)>0$. Then, the functional equation 
\begin{equation} \label{eq514.31}
\Gamma(z+1)=z\Gamma(z),
\end{equation}
see Lemma \ref{lem514.43}, or \cite[Equation 5.5.1]{DLMF}, provides an analytic continuation expression
\begin{eqnarray} \label{eq514.34}
\frac{\Gamma(1-z)\Gamma(z)}{\pi} 
&=&\frac{1}{\sin \pi z} \nonumber,
\end{eqnarray}
for any complex number $z \in \C$ such that $z \ne 0, -1, -2, -3, \ldots$. By Theorem \ref{thm938.40},  
\begin{eqnarray} \label{eq514.36}
\frac{1}{\sin |\pi z|} &\asymp &\frac{1}{\sin \left (\frac{1}{| \pi z|}\right )} \nonumber\\
&\asymp & |\pi z| \nonumber
\end{eqnarray}
as $|\pi z| \to \infty$. 
\end{proof}
\section{Basic Diophantine Approximations Results} \label{s2000}
All the materials covered in this section are standard results in the literature, see \cite{HW08}, \cite{LS95}, \cite{NZ91}, \cite{RH94}, \cite{SJ05}, \cite{WM00}, et alii. 
\subsection{Rationals And Irrationals Numbers Criteria} 
A real number \(\alpha \in \mathbb{R}\) is called \textit{rational} if \(\alpha = a/b\), where \(a, b \in \mathbb{Z}\) are integers. Otherwise, the number
is \textit{irrational}. The irrational numbers are further classified as \textit{algebraic} if \(\alpha\) is the root of an irreducible polynomial \(f(x) \in
\mathbb{Z}[x]\) of degree \(\deg (f)>1\), otherwise it is \textit{transcendental}.\\

\begin{lem} \label{lem2000.01} If a real number \(\alpha \in \mathbb{R}\) is a rational number, then there exists a constant \(c = c(\alpha )\) such that
\begin{equation}
\frac{c}{q}\leq \left|  \alpha -\frac{p}{q} \right|
\end{equation}
holds for any rational fraction \(p/q \neq \alpha\). Specifically, \(c \geq  1/b\text{ if }\alpha = a/b\).
\end{lem}

This is a statement about the lack of effective or good approximations for any arbitrary rational number \(\alpha \in \mathbb{Q}\) by other rational numbers. On the other hand, irrational numbers \(\alpha \in \mathbb{R}-\mathbb{Q}\) have effective approximations by rational numbers. If the complementary inequality \(\left|  \alpha -p/q \right| <c/q\) holds for infinitely many rational approximations \(p/q\), then it already shows that the real number \(\alpha \in \mathbb{R}\) is irrational, so it is sufficient to prove the irrationality of real numbers.

\begin{lem}[Dirichlet]\label{lem2000.02} 
 Suppose $\alpha \in \mathbb{R}$ is an irrational number. Then there exists an infinite
sequence of rational numbers $p_n/q_n$ satisfying
\begin{equation}
0 < \left|  \alpha -\frac{p_n}{q_n} \right|< \frac{1}{q_n^2}
\end{equation}
for all integers \(n\in \mathbb{N}\).
\end{lem}

\begin{lem} \label{lem2000.03}   Let $\alpha=[a_0,a_1,a_2, \ldots]$ be the continued fraction of a real number, and let $\{p_n/q_n: n \geq 1\}$ be the sequence of convergents. Then
\begin{equation}
0 < \left|  \alpha -\frac{p_n}{q_n} \right|< \frac{1}{a_{n+1}q_n^2}
\end{equation}
for all integers \(n\in \mathbb{N}\).
\end{lem}
This is standard in the literature, the proof appears in \cite[Theorem 171]{HW08}, \cite[Corollary 3.7]{SJ05}, \cite[Theorem 9]{KA97}, and similar references.\\

\begin{lem} \label{lem2000.05}   Let $\alpha=[a_0,a_1,a_2, \ldots]$ be the continued fraction of a real number, and let $\{p_n/q_n: n \geq 1\}$ be the sequence of convergents. Then
\begin{multicols}{2}
 \begin{enumerate} [font=\normalfont, label=(\roman*)]
\item$ \displaystyle 
 \frac{1}{2q_{n+1}q_n} \leq \left | \alpha - \frac{p_n}{q_n}  \right | \leq \frac{1}{q_n^{2}} 
$,
\item$\displaystyle
\frac{1}{2a_{n+1}q_n^2} \leq \left | \alpha - \frac{p_n}{q_n}  \right | \leq \frac{1}{q_n^{2}} 
$,
\end{enumerate}
\end{multicols}
for all integers \(n\in \mathbb{N}\).
\end{lem}
The recursive relation $q_{n+1}=a_{n+1}q_n+q_{n-1}$ links the two inequalities. Confer \cite[Theorem 3.8]{OC63}, \cite[Theorems 9 and 13]{KA97}, et alii. The proof of the best rational approximation stated below, appears in \cite[Theorem 2.1]{RH94}, and \cite[Theorem 3.8]{SJ05}. 
\begin{lem} \label{lem2000.07}   Let $\alpha \in \R$ be an irrational real number, and let $\{p_n/q_n: n \geq 1\}$ be the sequence of convergents. Then, for any rational number $p/q \in \Q^{\times}$, 
\begin{multicols}{2}
 \begin{enumerate} [font=\normalfont, label=(\roman*)]
\item$ \displaystyle 
 \left | \alpha q_n - p_n  \right | \leq  \left | \alpha q -p  \right |
$,
\item$\displaystyle
 \left | \alpha - \frac{p_n}{q_n}  \right | \leq  \left | \alpha - \frac{p}{q}  \right |
$,
\end{enumerate}
\end{multicols}

for all sufficiently large \(n\in \mathbb{N}\) such that $q \leq q_n$.
\end{lem}

\subsection{ Irrationalities Measures }
 
The concept of measures of irrationality of real numbers is discussed in \cite[p.\ 556]{WM00}, \cite[Chapter 11]{BB87}, et alii. This concept can be approached from several points of views. 

\begin{dfn} \label{dfn2000.01} {\normalfont The irrationality measure $\mu(\alpha)$ of a real number $\alpha \in \R$ is the infimum of the subset of  real numbers $\mu(\alpha)\geq1$ for which the Diophantine inequality
\begin{equation} \label{eq597.36}
  \left | \alpha-\frac{p}{q} \right | \ll\frac{1}{q^{\mu(\alpha)} }
\end{equation}
has finitely many rational solutions $p$ and $q$. Equivalently, for any arbitrary small number $\varepsilon >0$
\begin{equation} \label{eq597.36}
  \left | \alpha-\frac{p}{q} \right | \gg\frac{1}{q^{\mu(\alpha)+\varepsilon} }
\end{equation}
for all large $q \geq 1$.
}
\end{dfn}
\begin{thm} \label{thm2000.33} {\normalfont  (\cite[Theorem 2]{BY08}) } The map $\mu : \mathbb{R} \longrightarrow [2,\infty) \cup \{1\}$ is surjective function. Any number in the set $[2, \infty) \cup \{1\}$ is the irrationality measure  of some irrational number.
\end{thm} 

\begin{exa} \label{ex2000.33} {\normalfont Some irrational numbers of various  irrationality measures.

\begin{enumerate} [font=\normalfont, label=(\arabic*)]
\item  A rational number has an irrationality measure of $\mu(\alpha)=1$, see \cite[Theorem 186]{HW08}.
\item   An algebraic irrational number has an irrationality measure of $\mu(\alpha)=2$, an introduction to the earlier proofs of  Roth Theorem appears 

in \cite[p.\ 147]{RH94}.
\item   Any irrational number has an irrationality measure of $\mu(\alpha)\geq 2$.
\item   A Champernowne number $\kappa_b=0.123 \cdots b-1\cdot b \cdot b+1 \cdot b+2\cdots$ in base $b\geq 2$, concatenation of the $b$-base integers, 

has an irrationality measure of $\mu(\kappa_b)=b$.
\item   A Mahler number $\psi_b=\sum_{n \geq 1} b^{-[\tau]^ n}$ in base $b\geq 3$ has an irrationality measure of $\mu(\psi_b)=\tau$, for any real 

number $\tau \geq 2$, see \cite[Theorem 2]{BY08}.
\item   A Liouville number $\ell_b=\sum_{n \geq 1} b^{-n!}$ parameterized by $b \geq 2$ has an irrationality measure of $\mu(\ell_b)=\infty$, see \cite[p.\ 208]{HW08}.

\end{enumerate}
}
\end{exa}
\begin{dfn} \label{dfn2000.03} {\normalfont A measure of irrationality $\mu(\alpha)\geq 2 $ of an irrational real number $\alpha \in \R^{\times}$ is a map $\psi:\N \rightarrow \R$ such that for any $p,q \in \N$ with $q\geq q_0$, 
\begin{equation} \label{eq2000.70}
\left | \alpha - \frac{p}{q}  \right | \geq \frac{1}{\psi(q)} .
\end{equation}
Furthermore, any measure of irrationality of an irrational real number satisfies $\psi(q) \geq \sqrt{5}q^{\mu(\alpha)}\geq \sqrt{5}q^2$.

}
\end{dfn}
\begin{thm} \label{thm2000.03} For all integers $p,q\in \N$, and $q \geq q_0$, the number $\pi$ satisfies the rational approximation inequality 
\begin{equation} \label{eq2000.75}
\left | \pi - \frac{p}{q}  \right | \geq \frac{1}{q^{7.6063}} .
\end{equation}
\end{thm}
\begin{proof}[\textbf{Proof}] Consult the original source \cite[Theorem 1]{SV08}. 
\end{proof}

\subsection{ Normal Numbers }
The earliest study was centered on the distribution of the digits in the decimal expansion of the number $\sqrt{2}=1.414213562373 \ldots$, which is known as the Borel conjecture. \\

\begin{dfn} \label{dfn2200.01} {\normalfont An irrational number $\alpha \in \R$ is a \textit{normal number in base} $10$ if any sequence of $k$-digits in the decimal expansion occurs with probability $1/10^k$.
}
\end{dfn}

\begin{lem} {\normalfont (Wall)} \label{lem2200.15} An irrational number $\alpha \in \R$ is a normal number in base $10$ if and only if the sequence $\{\alpha 10^n: n \geq 1\}$ is uniformly distributed modulo $1$.
\end{lem}

\section{Numerical Data}\label{SN534}
The maxima of the function $1/\sin x$ occur at the numerators $x=p_n$ of the sequence of convergents $p_n/q_n \longrightarrow \pi$. The first few terms of the sequence $p_n$, which is cataloged as $A046947$ in \cite{OEIS}, are: 
\begin{equation} \label{eqN5534.100}
\begin{split}
 \mathcal{N}_{\pi}= \{1, 3, 22, 333, 355, 103993, 104348, 208341, 312689, 833719, 1146408, 
4272943, \\ 5419351, 80143857, 165707065, 245850922, 411557987, \ldots\}
\end{split}
\end{equation}

\subsection{Data For The Irrationality Measure} A few values were computed to illustrate the prediction in Theorem \ref{thm597.21}. The numerators $p_n$ and the denominators $q_n$ are listed in OEIS A002485 and A002486 respectively. The numerical data and Theorem \ref{thm597.21} are very well matched. The values of the approximate irrationality measure $\mu_n(\alpha)\geq2$ of the irrational number $\alpha\ne0$ is defined by
\begin{equation}\label{eqN5534.200}
	\mu_n(\alpha)=-\frac{\log \left |\alpha-p_n/q_n\right | }{\log q_n}, 
\end{equation}
where $n\geq 2$. 
\begin{exa}{\normalfont A large convergent is used here to illustrate the calculations, using 50 digits accuracy in the computer algebra system SAGE. The 80th convergent $p_n/q_n$ is given by
\begin{enumerate}
\item[(a)] $ \displaystyle p_{80}=32265750565715036834586769616835078002345262,$
\item[(b)] $ \displaystyle q_{80}=10270507390207332847445984588758042509963443,$
\item[(c)] $ \displaystyle p_{90}=3062433294496982257532162387854374057879036650780
,$
\item[(d)] $ \displaystyle q_{90}=974802793416785521474303406201616853353780695273
.$
\end{enumerate} 
Likewise, the corresponding $90$th approximation of the irrationality measure is
\begin{eqnarray}\label{eqN5534.210}
\mu_{80}(\pi)&=&-\frac{\log \left |\pi-p_{80}/q_{80}\right | }{\log q_{80}}\\&=&2.0006478846071613501149667587492035789605649634148
\nonumber. 
\end{eqnarray}
The corresponding $90$th approximation of the irrationality measure is
\begin{eqnarray}\label{eqN5534.220}
\mu_{90}(\pi)&=&-\frac{\log \left |\pi-p_{90}/q_{90}\right | }{\log q_{90}}\\
&=&2.0030433624458326071981396152293441039867289426019\nonumber.
\end{eqnarray}
}
\end{exa}

The range of values for $n\leq 25$ is plotted in Figure \ref{TN5534.200}. 
\begin{figure}[h!]
	\caption{Approximate Irrationality Measure $\mu_n(\pi)$ Of The Number $\pi.$} \label{TN5534.200}\centering%
	\begin{tikzpicture}
		\begin{axis}[
			xlabel=$n$,
			ylabel=$\mu_n(\pi)$,
			width=0.95\textwidth,
			height=0.5\textwidth		]
			\addplot[color=red,mark=square] coordinates {
				(2,3.429288)
				(3,2.014399)
				(4,3.201958)				
				(5,2.043905)				
				(6,2.096582)				
				(7,2.055815)				
				(8,2.107950)				
				(9,2.039080)				
				(10,2.120203)				
				(11,2.020606)				
				(12,2.189381)				
				(13,2.057220)				
				(14,2.044100)				
				(15,2.040522)				
				(16,2.058941)				
				(17,2.052652)				
				(18,2.049381)				
				(19,2.057213)				
				(20,2.203610)				
				(21,2.196409)				
				(22,2.034261)				
				(23,2.032066)				
				(24,2.019525)				
				(25,2.096513)	
			};
			\addplot[color=black,mark=square] coordinates {
				(2,2)
				(3,2)
				(4,2)
				(5,2)
				(6,2)
				(7,2)
				(8,2)
				(9,2)
				(10,2)
				(11,2)
				(12,2)
				(13,2)
				(14,2)
				(15,2)
				(16,2)
				(17,2)
				(18,2)
				(19,2)
				(20,2)
				(21,2)
				(22,2)
				(23,2)
				(24,2)
				(25,2)
			};
		\end{axis}
	\end{tikzpicture}	
\end{figure}
\begin{table}[h!]
\centering
\caption{Numerical Data For Irrationality Measure $|p_n/q_n-\pi|\geq q_n^{\mu_n(\pi)}$.} \label{t9007}
\begin{tabular}{l|l|l| l}
$n$&$p_n$&$q_n$&$\mu_n(\pi)$\\
\hline
1&$3$&   $1$   &$ $\\
2&$22$&  $7$   &$3.429288$\\
3&$333$&   $106$   &$2.014399$\\
4&$355$&  $113$   &$3.201958$\\
5&$103993$&   $33102$   &$2.043905$\\
6&$104348$&  $33215$   &$2.096582$\\
7&$208341$&  $66317$   &$2.055815$\\
8&$312689$&   $99532$   &$2.107950$\\
9&$833719$&  $265381$   &$2.039080$\\
10&$1146408$&  $364913$   &$2.120203$\\ 
11&$4272943$&      $1360120$   &$2.020606$ \\ 
12&$5419351$&   $1725033$   &$2.189381$\\
13&$80143857$&   $25510582$   &$2.057220$\\
14&$165707065$& $52746197$ & $2.044100$ \\
15&$245850922$&  $78256779$ & $2.040522$ \\ 
16&$411557987$&  $131002976$ & $2.058941$ \\
17&$1068966896$&    $340262731$ & $2.052652$ \\
18&$2549491779$&  $811528438$ & $2.049381$ \\
19&$6167950454$&  $1963319607$ & $2.057213$ \\
20&$14885392687$&  $473816765$ & $2.203610$ \\
21&$21053343141$&  $6701487259$ & $2.196409$ \\
22&$1783366216531$&  $567663097408$ & $2.034261$ \\
23&$3587785776203$&  $1142027682075$ & $2.032066$ \\
24&$5371151992734$&  $1709690779483$&$2.019525$\\
25&$8958937768937$&  $2851718461558$& $2.096513$\\
\end{tabular}
\end{table}

\subsection{The Sine Asymptotic Expansions} A few values were computed in Table \ref{t9105}  to illustrate the prediction in Theorem \ref{thm938.40}. Note that \eqref{eq534.111} is the restriction to convergents, so never vanishes, it is bounded below by $1/p_n$.
The numerical data in Table \ref{t9105} confirms this result.

\begin{table}[h!]
\centering
\caption{Numerical Data For $1/\sin p_n$} \label{t9105}
\begin{tabular}{l|l|l| l|l}
$n$&$p_n$&$1/\sin(p_n)$&$1/\sin 1/p_n$&$\sin p_n/\sin(1/p_n)$\\
\hline
1&$3$&   $7.086178$   &$3.0562$&$0.431303$\\
2&$22$&  $ -112.978$   &$22.0076$&$-0.194796$\\
3&$333$&   $-113.364$   &$333.001   $&$-2.93745$\\
4&$355$&  $-33173.7$   &$355.0$&$-0.0107013$\\
5&$103993$&   $-52275.7$   &$103993.0$&$-1.98932$\\
6&$104348$&  $-90785.1$   &$104348.0$&$-1.1494$\\
7&$208341$&  $123239.0$   &$208341.$&$1.69055$\\
8&$312689$&   $344744.0$   &$312689.0$&$0.907017$\\
9&$833719$&  $432354.0$   &$833719.0$&$1.92832$\\
10&$1146408$&  $-1.70132\times10^6$   &$1.14641\times10^6$&$-0.673836$\\ 
11&$4272943$&      $-1.81957\times10^6$   &$4.27294 \times10^6$&$-2.34832$\\ 
12&$5419351$&   $-2.61777\times10^7$   &$5.41935\times10^6$&$-0.207022$\\
13&$80143857$&   $-6.76918\times10^7$   &$8.01439\times10^7$&$-1.18395$\\
14&$165707065$& $-1.15543 \times10^8$ & $1.65707\times10^8$&$-1.43416$ \\
15&$245850922$&  $1.6345\times10^8$ & $2.45851\times10^8$&$1.50413$ \\ 
16&$411557987$&  $3.9421\times10^8$ & $4.11558\times10^8$ &$1.04401$\\
17&$1068966896$&    $9.57274\times10^8$ & $1.06897\times10^9$&$1.11668$ \\
18&$2549491779$&  $2.23489\times10^9$ & $2.54949\times10^9$&$1.14077$ \\
19&$6167950454$&  $6.6785\times10^{9}$ & $6.16795\times10^9$ &$0.923554$\\
20&$14885392687$&  $6.75763\times10^{9}$ & $1.48854\times10^{10}$&$2.20275$\\
21&$21053343141$&  $5.70327\times10^{11}$ & $2.10533\times10^{10}$ &$0.0369145$\\
22&$1783366216531$&  $1.43483\times10^{12}$ & $1.78337\times10^{12}$ &$1.24291$\\
23&$3587785776203$&  $2.78176\times10^{12}$ & $3.58779\times10^{12}$& $1.28975$\\
24&$5371151992734$&  $-2.96328\times10^{12}$ & $5.37115\times10^{12}$& $-1.81257$\\
25&$8958937768937$&  $-4.54136\times10^{13}$ & $8.95894\times10^{12}$& $-0.197274$\\
\end{tabular}
\end{table}

\subsection{The Sine Reflection Formula} 
Observe that the substitution $z \longrightarrow p_n/\pi$ in $\Gamma(1-z)\Gamma(z)$ leads to
\begin{equation} \label{eq534.99}
\frac{\Gamma(1-z)\Gamma(z)}{z}=\frac{\pi}{z\sin \pi z}=\frac{\pi^2}{p_n\sin p_n}=O(1).
\end{equation}
A few values were computed to illustrate the prediction in Lemma \ref{lem514.25}.  The ratio \eqref{eq534.99} is tabulated in the third column of Table \ref{t9009}. 

\begin{table}[h!]
\centering
\caption{Numerical Data For $\pi/\sin p_n$} \label{t9009}
\begin{tabular}{l|l|l| l}
$n$&$p_n$&$\Gamma(1-p_n/\pi)\Gamma(p_n/\pi)$&$\pi^2/p_n\sin p_n$\\
\hline
1&$3$&   $22.2619$   &$23.3126$\\
2&$22$&  $-354.93$   &$-50.6838$\\
3&$333$&   $-356.143$   &$-3.35992$\\
4&$355$&  $-104218.0$   &$-922.286$\\
5&$103993$&   $-164229.0$   &$-4.9613$\\
6&$104348$&  $-285210.0$   &$-8.58678$\\
7&$208341$&  $387167.0$   &$5.83812$\\
8&$312689$&   $1.08305\times10^6$   &$10.8814$\\
9&$833719$&  $1.35828\times10^6$   &$5.11823$\\
10&$1146408$&  $-5.34484\times10^6$   &$-14.6469$\\ 
11&$4272943$&      $-5.71636\times10^6$   &$-4.20283$ \\ 
12&$5419351$&   $-8.22396\times10^7$   &$-47.6742$\\
13&$80143857$&   $-2.1266\times10^8$   &$-8.33615$\\
14&$165707065$& $-3.62989 \times10^8$ & $-6.88181$ \\
15&$245850922$&  $5.13494\times10^8$ & $6.56166$ \\ 
16&$411557987$&  $1.23845\times10^9$ & $9.45359$ \\
17&$1068966896$&    $3.00736\times10^9$ & $8.83836$ \\
18&$2549491779$&  $7.02111\times10^9$ & $8.65171$ \\
19&$6167950454$&  $2.09811\times10^{10}$ & $10.6866$ \\
20&$14885392687$&  $2.12297\times10^{10}$ & $4.48058$ \\
21&$21053343141$&  $1.79173\times10^{12}$ & $267.364$ \\
22&$1783366216531$&  $4.50764\times10^{12}$ & $7.9407$ \\
23&$3587785776203$&  $8.73917\times10^{12}$ & $7.65233$ \\
24&$5371151992734$&  $-9.30941\times10^{12}$&$-5.44508$\\
25&$8958937768937$&  $1.42671\times10^{14}$& $-50.0299$\\
\end{tabular}
\end{table}

\newpage
\section{Problems} \label{pr70000}
\subsection{Flint Hill Class Series}
\begin{exe} {\normalfont
		The infinite series $$ S_1=\sum_{n \geq 1} \frac{(-1)^{n+1}}{n}\qquad \text{ and } \qquad  S_2=\sum_{n \geq 1} \frac{1}{n^{2} \sin n}$$ 
		have many asymptotic similarities. The first is conditionally convergent. Is the series $S_2$ conditionally convergent?}
\end{exe}
\begin{exe} {\normalfont
		Show that the infinite series $$ S_3=\sum_{n \geq 1} \frac{\cos^{v}n}{n}\qquad \text{ and } \qquad  S_4=\sum_{n \geq 1} \frac{\sin^v n}{n}$$ 
		are conditionally convergent for odd $v \geq 1$. Are these series absolutely conditionally convergent for any $v>0$?}
\end{exe}
\begin{exe} {\normalfont Explain the convergence of the infinite series 
		$$ \frac{\pi}{\sin \pi z}=\frac{1}{z}+2z\sum_{n \geq 1} \frac{(-1)^{n}}{z^2-n^2}$$ 
		at the real numbers $z=n/\pi$, where $n \geq 1$ is an integer.  Is this series convergent?}
\end{exe}
\begin{exe} {\normalfont Explain the convergence of the infinite series 
		$$ \pi \cot \pi z=\frac{1}{z}+2z\sum_{n \geq 1} \frac{1}{z^2-n^2}$$ 
		at the real numbers $z=n/\pi$, where $n \geq 1$ is an integer. Is this series convergent? The derivation of this series appears in \cite[p.\ 122]{CJ90}.}
\end{exe}

\begin{exe} {\normalfont Explain the convergence of the infinite series
		$$ \frac{\pi}{\sin^2 \pi z}=\frac{1}{\pi^2}\sum_{-\infty<n<\infty} \frac{1}{(z-n)^2}$$ 
		at the real numbers $z=n/\pi$, where $n \geq 1$ is an integer.  Is this series convergent? The derivation of this series appears in \cite[p.\ 122]{CJ90}, \cite[p.\ 276]{GK06}. }
\end{exe}

\begin{exe} {\normalfont Explain the convergence of the infinite series$$ \frac{\sin z}{\sin \pi z}=\frac{2}{\pi}\sum_{n \geq 1} \frac{(-1)^{n}\sin nz}{z^2-n^2}$$ 
		at the real numbers $z=n/\pi$, where $n \geq 1$ is an integer. This series has a different structure than the previous exercise, is this series  is absolutely convergent? conditionally convergent? The derivation of this series appears in \cite[p.\ 276]{GK06}. }
\end{exe}
\subsection{Complex Analysis}
\begin{exe} {\normalfont
		Assume that the infinite series is absolutely convergent. Use a Cauchy integral formula to evaluate the integral
		$$\frac{1}{i2\pi} \int_{C} \frac{1}{z^u \sin^v z} dz,$$
		where $C$ is a suitable curve.} 
\end{exe}
\begin{exe} {\normalfont
		Determine whether or not there is a complex valued function $f(z)$ for which the Cauchy integral evaluate to
		$$\frac{1}{i2\pi} \int_{C} f(z) dz=\sum_{n \geq 1} \frac{1}{n^{u}\sin^v n},$$
		where $C$ is a suitable curve, and the infinite series is absolutely convergent.} 
\end{exe}
\begin{exe} {\normalfont
		Let $z \in \C$ be a large complex variable. Prove or disprove the asymptotic relation   
		$$ \label{eq514.36}
		\frac{1}{\sin |\pi z|} \asymp \frac{1}{\sin \left (\frac{1}{| \pi z|}\right )} \asymp  |\pi z| \nonumber\\.
		$$} 
\end{exe}
\begin{exe} {\normalfont Consider the product 
		\begin{equation} \label{eq934.99}
			\sin \left (\frac{1}{z}\right ) \times \sin z=1+O(1/z^2)
		\end{equation}
		of a complex variable $z \in \C$. Explain its properties in the unit disk $\{z \in \C: |z|<1 \}$ and at infinity. 
	}
\end{exe}

\begin{exe} {\normalfont
		Let $z \in \C$ be a large complex variable. Use the properties of the gamma function to prove the asymptotic relation   
		$$ 
		\left | \Gamma(1-z)\Gamma(z) \right | =O(|z|).
		$$} 
\end{exe}
\begin{exe} {\normalfont
		Let $x \in \R$ be a large real variable. Use the trigamma reflection formula $$\psi_1(1-x)+\psi_1(x)=\frac{\pi^2}{\sin^2 \left (\pi x\right )}$$ to derive an asymptotic relation   
		$$ \label{eq514.836}
		\frac{1}{\sin^2 |\pi x|}  \asymp  |\pi x|^2\nonumber\\.
		$$Explain any restrictions on the  real variable $x \in \R$. } 
\end{exe}

\subsection{Irrationality Measures}
\begin{exe} {\normalfont Do the irrationality measures of the numbers $\pi$ and $\pi^2$ satisfy $\mu(\pi)=\mu(\pi^2)=2$? This is supported by the similarity of $\sin(n)=\sin(n-m\pi)$ and $$ \frac{\sin x}{x}=\prod_{m \geq 1} \left ( 1-\frac{x^2}{\pi^2 m^2} \right )$$ at $x=n$. More generally, $x=n^k \pi^{-k+1}$ } 
\end{exe}

\begin{exe} {\normalfont
		What is the relationship between the irrationality measures of the numbers $\pi$ and $\pi^k$, for example, do these measures satisfy $\mu(\pi)=\mu(\pi^k)=2$ for $k\geq 2$?}
\end{exe}

\subsection{Partial Quotients}
\begin{exe} {\normalfont
		Determine an explicit bound $a_n \leq B$ for the continued fraction of the irrational number $\pi=[a_0;a_1,a_2, a_3,\ldots]$ for all partial quotients $a_n$ as $n \to \infty$.}
\end{exe}
\begin{exe} {\normalfont
		Does $\pi^2=[a_0;a_1,a_2, a_3,\ldots]$ has bounded partial quotients $a_n$ as $n \to \infty$?}
\end{exe}

\begin{exe} {\normalfont Let $\alpha=[a_0;a_1,a_2, a_3,\ldots]$ and  $\beta=[b_0;b_1,b_2, b_3,\ldots]$ be a pair of continued fractions. Is there an algorithm to determine whether or not $\alpha=\beta$ or $\alpha \ne \beta$ based on sequence of comparisons $a_0=b_0, a_1=b_1,  a_2=b_2, \ldots, a_N=b_N$ for some $N\geq 1$?}
\end{exe}

\begin{exe} {\normalfont Let $e=[a_0;a_1,a_2, a_3,\ldots]$, where $a_{3k}=a_{3k+2}=1$, and $a_{3k+1}=2k$, be the of continued fraction of the natural base $e$. Compute the geometric mean value of the partial quotients:
		$$ K_e=\lim_{n \to \infty} \left (\prod_{k \leq n} a_k\right )^{1/n}=\infty ?$$
	}
\end{exe}
\begin{exe} {\normalfont Let $\pi=[a_0;a_1,a_2, a_3,\ldots]$ be the of continued fraction of the natural base. A short calculation give 
		$$K_{\pi}=\lim_{n \leq 10} \left (\prod_{k \leq n} a_k\right )^{1/n}=3.361 \ldots, \, K_{\pi}=\lim_{n \leq 20} \left (\prod_{k \leq n} a_k\right )^{1/n}=2.628\ldots.  $$
		Compute a numerical approximation for the geometric mean value of the partial quotients:
		$$K_{\pi}=\lim_{n \leq 1000} \left (\prod_{k \leq n} a_k\right )^{1/n}=?$$
	}
\end{exe}

\begin{exe} {\normalfont Let $\alpha=[a_0;a_1,a_2, a_3,\ldots]$ be the of continued fraction of a real number. Assume it has a lacunary subsequence of unbounded  partial quotients $a_{n_i}=O(\log \log n)$, such that $n_{i+1}/n_{i}>1$, otherwise $a_{n}=O(1)$. Is the geometric mean value of the partial quotients unbounded
		$$ K_{\alpha}=\lim_{n \to \infty} \left (\prod_{k \leq n} a_k\right )^{1/n}=\infty ?$$
	}
\end{exe}
\subsection{Concatenated Sequences}
\begin{exe} {\normalfont A Champernowne number $\kappa_b=0.123 \cdots b-1\cdot b \cdot b+1 \cdot b+2\cdots$ in base $b\geq 2$ is formed by concatenating the sequence of consecutive integers in base $b$ is irrationality.  
		Show that the number $0.F_0F_1F_2F_3 \ldots =1/F_{11}$ formed by concatenating the sequence of Fibonacci numbers $F_{n+1}=F_n+F_{n-1}$  is rational.}
\end{exe}
\begin{exe} {\normalfont Let $f(x) \in \Z[x]$ be a polynomial, and let $D_n=|f(n)|$. Show that the number $0.D_0D_1D_2D_3 \ldots $ formed by concatenating the sequence of values is irrational.}
\end{exe}
\begin{exe} {\normalfont Let $\{D_n\geq 0: n\geq 0\}$ be an infinite sequence of integers, and let $\alpha=0.D_0D_1D_2D_3 \ldots $ formed by concatenating the sequence of integers. Determine a sufficient condition on the sequence of integers to have an irrational number $\alpha>0$.}
\end{exe}

\subsection{Exact Evaluations Of Power Series}
\begin{exe} {\normalfont The exact evaluation of the first series below is quite simple, but the next series requires some work:
$$L_1(1/2)=\sum_{n \geq1} \frac{1}{2^nn}=\log 2 \quad \text{ and }\quad L_2(1/2)=\sum_{n \geq1} \frac{1}{2^nn^2}=\frac{\pi^2}{12}-\frac{1}{2}\left (\log 2\right )^2. $$
Prove this, use the properties of the polylogarithm function $L_k(z)=\sum_{n \geq 1}z^n/n^k$.} 
\end{exe}
\begin{exe} {\normalfont
		Determine whether or not the series
		$$L_3(1/2)=\sum_{n \geq1} \frac{1}{2^nn^3}$$
		has a closed form evaluation (exact).} 
\end{exe}

\subsection{Flat Hill Class Series}
\begin{exe} \label{exe79000.01}{\normalfont Let $\alpha=(1+\sqrt{5})/2$ be a Pisot number. Show that $||\alpha^n||  \to 0$ as $n \to \infty$. Explain the convergence of the first infinite series 
		$$  \sum_{n \geq 1} \sin|| \alpha^n||  \quad \text{ and } \quad \sum_{n \geq 1} \frac{1}{n^2 \sin|| \alpha^n||}$$ 
		the divergence of the second infinite series.}
\end{exe}
\begin{exe} \label{exe79000.05}{\normalfont Assume that $\pi$ is a normal number base $10$. Show that $||\pi 10^n||$  is uniformly distributed. Explain the convergence or divergence of the infinite series 
		$$
		\sum_{n \geq 1} \frac{1}{n^2 \sin|| \pi 10^n||}$$ .
	}
\end{exe}

\begin{exe} \label{exe79000.10}{\normalfont Assume that $\pi$ is a normal number base $10$. Determine the least parameters $a>1$ and $b>0$  for which the infinite series 
		$$
		\sum_{n \geq 1} \frac{1}{n^a \sin^b|| \pi ^n||}$$ 
		converges. For example, is $a-b>0$ sufficient?
	}
\end{exe}

\begin{exe} \label{exe79000.15}{\normalfont Assume that $\pi$ is a normal number base $10$. Determine the least parameters $a>1$ and $b>0$  for which the infinite series 
		$$
		\sum_{n \geq 1} \frac{1}{n^a \sin^b|| \pi 10^n||}$$ 
		converges. For example, is $a-b>0$ sufficient?
	}
\end{exe}


\currfilename\\

\end{document}